\documentclass[11pt]{article}
\usepackage[a4paper,margin=3cm]{geometry}
\usepackage[utf8]{inputenc} 
\usepackage[T1]{fontenc} 
\usepackage{lmodern}
\usepackage{todonotes}
\usepackage[autolanguage]{numprint} 
\usepackage{hyperref} 
\usepackage{graphicx} 
\usepackage[all]{xy}
\usepackage{amsmath,amssymb,amsthm}

\newcommand\NN{\mathbb{N}} 
\newcommand\RR{\mathbb{R}} 
\newcommand\CC{\mathbb{C}} 
\newcommand\PP{\mathbb{P}} 
\newcommand\EE{\mathbb{E}} 

\newtheorem{theorem}{Theorem}[section]%
\newtheorem{corollary}[theorem]{Corollary}%
\newtheorem{lemma}[theorem]{Lemma}%

\author{Raphaël Butez\footnote{ butez@ceremade.dauphine.fr, CEREMADE, UMR CNRS 7534 Universit\'{e} Paris-Dauphine, PSL Research university, Place du Mar\'{e}chal de Lattre de Tassigny 75016 Paris, FRANCE.}}

\title{The largest root of random Kac polynomials is heavy tailed}

\begin{document} 

\maketitle 

\begin{abstract}
We prove that the largest and smallest root in modulus of random Kac polynomials have a non-universal behavior. They do not converge towards the edge of the support of the limiting distribution of the zeros. This non-universality is surprising as the large deviation principle for the empirical measure is universal. This is in sharp contrast with random matrix theory where the large deviation principle is non-universal but the fluctuations of the largest eigenvalue are universal. We show that the modulus of the largest zero is heavy tailed, with a number of finite moments bounded from above by the behavior at the origin of the distribution of the coefficients. We also prove that the random process of the roots of modulus smaller than one converges towards a limit point process. Finally, in the case of complex Gaussian coefficients, we use the work of Peres and Vir\'ag \cite{peresvirag} to obtain explicit formulas for the limiting objects.
\end{abstract}

\section{Introduction}
Consider a random polynomial of the form:
\[ P_n(z)= \sum_{k=0}^n a_k z^k = a_n \prod_{k=1}^{n}(z-z_k^{(n)}) \]
where $a_0, \dots, a_n$ are i.i.d. random variables and $z_1^{(n)}, \dots, z_n^{(n)}$ are the complex zeros of $P_n$. These polynomials are often called Kac polynomials. The zeros of these polynomials are known to concentrate on the unit circle of $\CC$ as their degree tends to infinity under some moment condition on the coefficients. This universal behavior has been studied by many authors since the work of Sparo and Shur \cite{sparosur} and we refer to the book \cite{bharucha} for more precise information on the history of the topic. The most precise result about this convergence was given by Ibragimov and Zaporozhets in \cite{ibragimovzaporozhets} where they prove that for any bounded and continuous function $f$ and any $\varepsilon >0$
\[ \PP\left( \left| \frac{1}{n}\sum_{k=1}^n f(z_k^{(n)}) - \frac{1}{2\pi}\int_0^{2\pi} f(e^{i\theta})d\theta \right| > \varepsilon \right)\xrightarrow[n \to \infty]{} 0\]
if and only if $\EE(\log(1+|a_0|))<\infty$ and $\PP(a_0=0)<1$.

This result means that a proportion going to one of the zeros cluster uniformly on the unit circle. It does not prevent a negligible part of the zeros to be real or to be away from the unit circle. In this situation, it is natural to ask if $\max_k|z_k^{(n)}|$ and $\min_k |z_k^{(n)}|$ converge towards $1$ as $n$ goes to infinity. In this note we prove that the behavior of the extremal zeros of $P_n$ is not universal and that the random variable $\max |z_k^{(n)}|$ is usually a heavy tailed random variable. We give an upper bound on the number of finite moments depending on the cumulative distribution function of the coefficients at $0$.

To study the zeros of $P_n$ for large $n$, we may want to see $P_n$ as the partial sum of a random entire series.
If we assume that the random variable $|a_0|$ is non-deterministic and satisfies $\EE(\log(1+|a_0|))< \infty$, then the entire series $P_{\infty}(z)=\sum_{k=0}^{\infty}a_k z^k$ has almost surely a radius of convergence equal to $1$ and $P_{\infty}$ is a random non-constant holomorphic function on the unit disk. Hence $P_{\infty}$ has a countable set of zeros $\{z_k^{(\infty)}\}$ which are counted with multiplicity. As $P_{\infty}$ is almost surely a non-constant analytic function, its zeros are isolated and have no accumulation point inside the open unit disk. This ensures that there is a finite number of zeros inside any compact set in the disk.

\begin{theorem}[Main result]\label{heavytail} Assume that the random variable $a_0$ satisfies $\EE(\log(1+|a_0|))< \infty$, that $|a_0|$ is not deterministic and that $\PP(a_0=0)=0$. Let $n \in \NN \cup \{\infty \}$ and
\[ x_1^{(n)}= \min_{k} |z_k^{(n)}| \text{ and if $n<\infty$,  }x_n^{(n)}= \max_{k} |z_k^{(n)}|.  \]
\begin{enumerate}

\item \label{point 1} The random variable $x_n^{(n)}$ has the same distribution as $1/x_1^{(n)}$.

\item There exists three constants $C_1>0$ , $r>0$ and $A>0$ depending only on the distribution of $|a_0|$ such that
\begin{equation}\label{keyinequality}
\forall \  0<t<C_1, \quad \PP\left(x_1^{(n)} \leq t \right) \geq \PP\left(|a_0|\leq \frac{rt}{2} \right) A.
\end{equation}

\item If there exists $k \geq 0$, $a>0$ and $\delta>0$ such that 
\[\forall \  t<\delta ,\quad \PP(|a_0| \leq t) \geq at^{k}  \quad \text{then} \quad  \EE((x_n^{(n)})^k)= \infty.\]

\item Almost surely, the point process $\chi_n=\left\{ z_k^{(n)}\mid |z_k^{(n)}|<1 \right\}$ converges weakly in the space of Radon measures towards $\chi_{\infty}= \left\{ z_k^{(\infty)} \right\}$. More precisely, for any continuous and compactly supported function $f$ defined on the open unit disk $D(0,1)$ we have
\[ \sum_k f(z_k^{(n)}) \xrightarrow[n \to \infty]{a.s} \sum_k f(z_k^{(\infty)}).  \]

\item The random variable $x_1^{(n)}$ converges almost surely towards $x_1^{(\infty)}$ and $x_n^{(n)}$ converges in distribution towards $x^{(\infty)}:=1/x_1^{(\infty)}$.
\end{enumerate}
\end{theorem}

\noindent The condition $\PP(a_0=0)=0$ ensures that the degree of $P_n$ is $n$. Hence the zero set $\{ z_k^{(n)}\}$ is well defined. If we remove this condition, the theorem is still true if we take the convention that for a polynomial of degree $k<n$, $z_{k+1}^{(n)} = \dots = z_{n}^{(n)} =  \infty$.

In the case of real Gaussian coefficients, Majumdar and Schehr proved a similar result to (3) for the largest real root in \cite{schehrmajumdar2008}. They use a Taylor expansion of the first intensity function of the real zeros at $0$ to prove that the density of $x_n^{(n)}$ decays at infinity like $1/t^2$.

Notice that the first three points of Theorem \ref{heavytail} are valid for any fixed $n$. The heavy tail behavior of the largest root is not asymptotic. The points (4) and (5) imply that the heavy tail phenomena do not vanish at infinity, namely that $x^{(\infty)}$ does not have more finite moments that $x_n^{(\infty)}$. The point (4) is a deterministic statement, and is a direct application of Hurwitz's theorem in complex analysis. It will allow us obtain Corollary \ref{gaussiancase} which gives the limit of the point process $\chi_n$ for complex Gaussian coefficients.

What does this theorem say for some classical distribution of the coefficients?
If the distribution of $|a_0|$ is absolutely continuous with respect to the Lebesgue measure on $\RR_+$ with density $g$, then the point (3) can be linked to the density $g$ at zero:

\begin{enumerate}
\item if $g$ is continuous at $0$ and $g(0)>0$ then the largest root in modulus has infinite mean;

\item if $g(t) \sim \alpha t$ at zero, then $x_n^{(n)}$ has an infinite second moment.
\end{enumerate}

\noindent These two situations cover most of the classical examples of random variables. Real Gaussian random variables, exponential random variables, Cauchy random variables and many others are covered by the first part of the remark. Radial complex random variables such as complex Gaussian random variables are covered by the second point. This is a consequence of the polar change of coordinates which adds a factor $2\pi r$ to the density. This phenomenon is illustrated in Figure \ref{figure1}: when the coefficients are complex, the density of $x_1^{(n)}$ vanishes at zero.

If the distribution of the $a_k$'s is supported in an annulus bounded away from zero, then all the roots lie in an annulus. This can be seen as a consequence of the Gershgorin circle theorem of localization of the eigenvalues of matrices, applied to the companion matrix of the polynomial $P_n$. In this setting, we also know the empirical measure of $P_n$ converges deterministically towards the uniform measure on the unit circle \cite{hughes}.

This theorem can be surprising if we compare it to similar results in random matrix theory. There is a strong analogy of results and techniques between random polynomials and random matrices. For Ginibre random matrices or Kac random polynomials, the empirical measures converge towards a deterministic measure and the explicit distribution of the eigenvalues (or zeros) can be computed when the coefficients are Gaussian (real or complex). Large deviation principles for the empirical measures were obtained with the same speed and very similar rate functions (\cite{hiaipetz}, \cite{benarouszeitouni} for Ginibre, \cite{zeitounizelditch} and \cite{butez} for Kac polynomials). For random matrices, the large deviation principle for the empirical measure is known not to be universal \cite{bordenavecaputo} and to depend on the tail of the coefficients of the matrix. For random polynomials, the large deviation principle is universal\cite{butezzeitouni}.
For random matrices, the fluctuations of the largest eigenvalue at the edge of the limiting distribution have been studied by many authors since \cite{tracywidom} and have proven to be universal.

When the $a_k$'s are i.i.d. standard complex Gaussian random variables, the zeros of $P_n$ form a Coulomb gas in $\CC^n$ (see \cite{zeitounizelditch}) with density of the form 
\[ (z_1^{(n)},\dots,z_n^{(n)}) \sim \frac{1}{Z_n} \exp \left( \sum_{i \neq j} \log |z_i^{(n)}-z_j^{(n)}| -(n+1) \log \frac{1}{2\pi}\int_0^{2\pi} \prod_{k=1}^{n}|e^{i\theta}-z_k^{(n)}|^2 d\theta \right). \]

\noindent This is similar to the eigenvalues of complex Ginibre random matrices where the density of the eigenvalues on $\CC^n$ is of the form (see \cite{ginibre})
\[(\lambda_1 , \dots, \lambda_n)\sim \frac{1}{Z_n'} \exp \left( \sum_{i\neq j} \log |\lambda_i-\lambda_j| - n \sum_{k=1}^n \frac{|\lambda_i|^2}{2}  \right). \]

They have in common the same interaction between the particles, but the confining term is different. Why does the largest particle have such a different behavior for polynomials and matrices? If we look at the behavior of the confining term in each variable for random polynomials, we see that it grows at infinity like $\log(|z|)$ while the confining term for Ginibre is $V(z)=|z|^2/2$. We believe that it is a general fact for Coulomb gases: when the confining term is of order $\log|z|$ at infinity, the largest particle has a heavy tail and when the confining term is stronger than logarithm, the largest particle should converge towards the edge of the limiting distribution.  The potential energy of the largest particle is approximatively $\int -\log|z-w|d\mu_{\infty}(w) + V(z)$, where $\mu_{\infty}$ is the limiting distribution of the particles. Hence, if $V$ grows faster than logarithm, the cost of having a particle far from the support of $\mu_{\infty}$ grows at infinity. On the other hand, if $V$ and $\int -\log|z-w|d\mu_{\infty}(w)$ are of same order, the cost of having a particle far from the support of $\mu_{\infty}$ is finite. To our knowledge, this phenomenon is not treated in the literature. The heuristic above only relies on ``energetic'' considerations and has not been proved so far.
The same phenomenon should appear for real Gaussian coefficients, as the distribution of the zeros of $P_n$ \cite{butez} is also very similar to the distribution of the eigenvalues of the real Ginibre ensemble \cite{edelman1997proba}: both form a mixture of Coulomb gases, each gas having a fixed number of particles on the real line.

If we compare the hypotheses of Theorem \ref{heavytail} with the one of the main theorem from \cite{butezzeitouni}, we see that both rely on the behavior of the distribution of the coefficients at zero. Within the universality class of random Kac polynomials with Gaussian coefficients, the number of finite moments for the largest root is constant.

\begin{figure}[h]
	\includegraphics[scale=0.54]{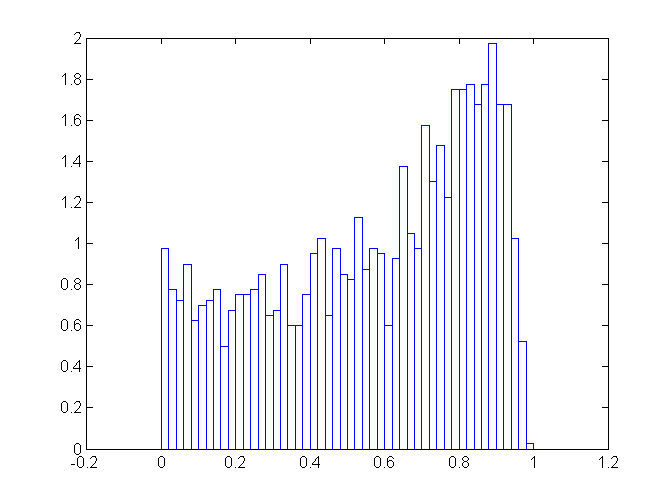}\hfill
	\includegraphics[scale=0.54]{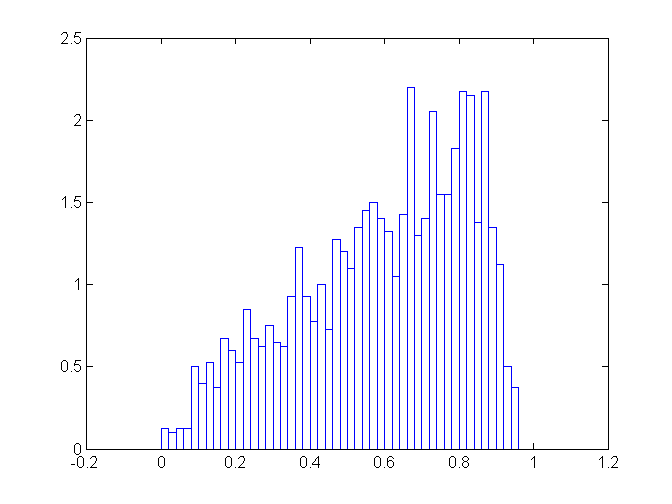}
	\caption{Histogram of the $x_1^{(500)}$ for exponential coefficients with mean $1$ (left) and $d\mu=1/2\pi e^{-|z| }dz$ on $\CC$ (right).}
	\label{figure1}
\end{figure}

\subsection*{Complex Gaussian coefficients.}
Theorem \ref{heavytail} states that $x_1^{(n)}$ and $x_n^{(n)}$ converge in distribution but do not give any precise information on the limit. In the case of complex Gaussian coefficient, we can give more precise results on this limiting distribution as the zeros of $P_{\infty}$ have been studied in the case of Gaussian Analytic Functions (GAF) by Peres and Vir\'ag in \cite{peresvirag}
\[P_{\infty}(z) = \sum_{k \in \NN}a_k z^k .\]
The corollary below is just a combination of Theorem \ref{heavytail} along with \cite[Theorem 5.1.1 and Corollary 5.1.7]{houghtkris}.
\begin{corollary}[Gaussian case]\label{gaussiancase}
Let $(a_k)_{k \in \NN}$ be a sequence i.i.d. standard complex Gaussian random variables.
\begin{enumerate}

\item The point process $\left\{ z_k^{(n)} \mid |z_k^{(n)}|<1  \right\}$ converges towards the determinantal point process in the open unit disk $D(0,1)$ with the Bergman Kernel
\begin{equation*}
\forall z , w \in D(0,1) ,\quad K(z,w)= \frac{1}{\pi(1-z\bar{w})^2}.
\end{equation*}
As a consequence, the set $\{|z_k^{(n)}|\}_{k \geq 1}$ has the same law as the set $\{U_k^{1/2k}\}_{k \geq 1}$ where $(U_k)_{k \geq 1}$ is a sequence of i.i.d. uniform random variables on $[0,1]$.

\item The smallest root in modulus of $P_{\infty}$, $z_1$, has a rotationally invariant distribution and its modulus, $x_1^{(\infty)}$, has a cumulative distribution function defined on $(0,1)$ given by:
\begin{equation*}
F_{x_1^{(\infty)}}(t)= 1 - \prod_{k=1}^{\infty}(1-t^{2k}).
\end{equation*}
\end{enumerate}

\end{corollary}

\noindent
This theorem is no more than the combination of the work of Peres and Vir\'ag \cite{peresvirag} with Theorem \ref{heavytail}. This allows us to compute the exact limit distribution of the smallest modulus of the zeros of $P_n$. Notice that the cumulative distribution function $F_{x_1^{(\infty)}}$ is of order $t^2$ around zeros, which implies that $1/x_1^{(\infty)}$ has an infinite variance.

Figure \ref{figure2} is an illustration of the point (2) of Corollary \ref{gaussiancase}. The histogram of $x_1^{(500)}$ is very similar to the graph of the density of $x_1^{\infty}$.

\begin{figure}[!h]

\includegraphics[scale=0.53]{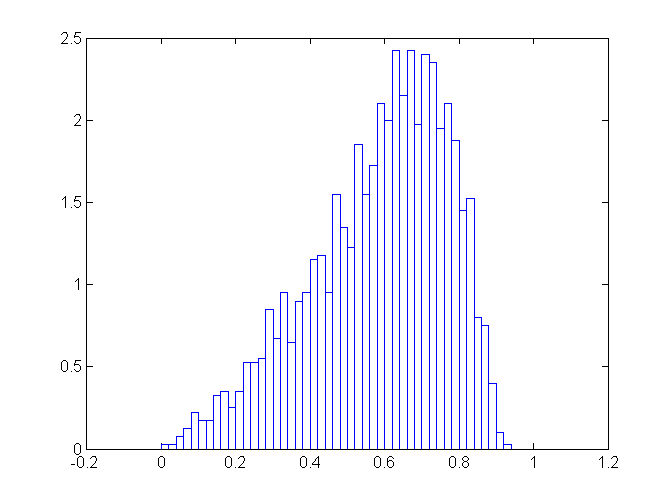}\hfill
\includegraphics[scale=0.54]{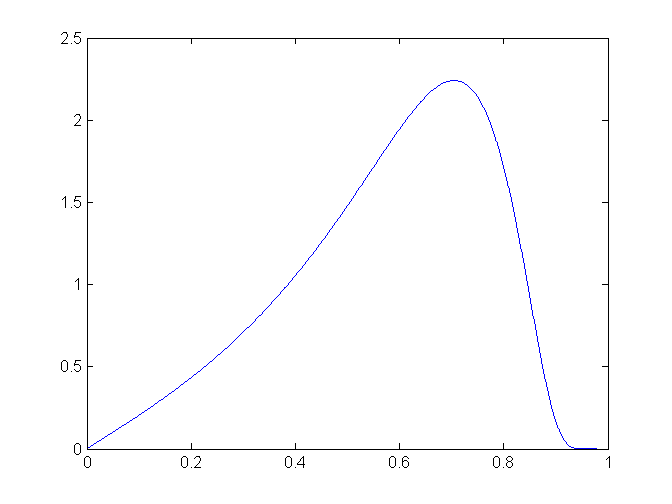}
\caption{Histogram of $x_1^{(500)}$ for complex Gaussian coefficients and density of $x_1^{(\infty)}$.  }
\label{figure2}
\end{figure}

\section{Proofs of the results}
We start with a lemma that will be essential in the proof of Theorem \ref{heavytail}.
\begin{lemma}[Lemma 4.1 in \cite{kabluchkozapo}] \label{lemma}
	Let $(a_k)_{k\in \NN}$ be i.i.d. random variables. Fix $\varepsilon > 0$. Then 
	\[\sup\limits_{k\in \NN} \frac{|a_k|}{e^{\varepsilon k}} < \infty    \text{  a.s.}   \Leftrightarrow \EE( \log (1+a_0)) < \infty.  \]
\end{lemma}

\begin{proof}[Proof of the lemma]
	For every non negative random variable $X$ we have:
	\[ \sum\limits_{k=1}^{\infty} \PP(X\geq k) \leq \EE(X) \leq \sum\limits_{k=0}^{\infty} \PP(X\geq k) . \]
	Those inequalities come from the relation: $\EE(X)= \int_{\RR^+} \PP(X\geq x) dx$.
	
\noindent Now we use this inequality to the non negative variable \[X= \frac{1}{\varepsilon} \log (1+|a_0|).\] We deduce that \[\sum\limits_{k=1}^{\infty} \PP(\frac{|a_k|}{e^{\varepsilon k}} > 1) < \infty\] and so we have, thanks to the Borel-Cantelli lemma \[\limsup \frac{|a_k|}{e^{\varepsilon k}} \leq 1,\] which implies $\sup\limits_{k\in \NN} \frac{|a_k|}{e^{\varepsilon k}} < \infty$.
	
	The reverse implication relies on a similar reasoning, and will not be used in the proof of the theorem.
\end{proof}
\begin{proof}[Proof of Theorem \ref{heavytail}] \textbf{Proof of 1.}
Let $P_n(z)= a_0+ a_1 z+\dots+a_nz^n$ and \[Q_n(z)=z^n P_n(1/z)= a_n = a_{n-1}z + \dots = a_0z^n.\] As the $a_k$ are i.i.d. random variables, the distribution of the random polynomials $P_n$ and $Q_n$ are the same. If $\{z_k^{(n)}\}$ is the set of zeros of $P_n$ then the set of zeros of $Q_n$ is $\{1/z_k^{(n)}\}$. This implies that $x_n^{(n)}$ and $1/x_1^{(n)}$ have the same distribution.

\paragraph{Proof of 2.} Fix $n \in \NN \cup \{\infty\}$. The random variable $M:=\sup_{k \geq 2} |a_k|/ e^k$ is almost surely finite. This is obvious for $n<\infty$ and this is a consequence of lemma \ref{lemma} for $n=\infty$. There exists $K$ such that $\PP(M<K)>0$. Let us define $C_2 = \PP(M<K)>0$.
The key idea of this proof is to use Rouché's theorem \cite[p.~181]{lang} to show that $P_n$ and $P_1(z)=a_0+a_1z$ have the same number of roots in a neighborhood of the origin. Rouché's theorem is the following: if $\gamma$ is a closed path holomogous to $0$ in some open set $\mathcal{U}$ such that $\gamma$ has an interior and $f$ and $g$ are two analytic functions on $\mathcal{U}$ such that for any $z \in \gamma$
\[ |f(z)-g(z)|<|f(z)| \]
then $f$ and $g$ have the same number of zeros in the interior of $\gamma$.

\noindent To bound from below the probability that $x_1^{(n)}$ is smaller than $t$, we compare it to the modulus of the root of $P_1$.
\begin{align*}
\PP(x_1^{(n)} \leq t) &\geq \PP\left(x_1^{(n)}   \leq 2\frac{|a_0|}{|a_1|} \text{ and } 2\frac{|a_0|}{|a_1|} \leq t\right) \\
& \geq \PP\left(P_n \text{ has exactly one zero in } B\left(0,2\frac{|a_0|}{|a_1|} \right) \text{ and } 2\frac{|a_0|}{|a_1|} \leq t \right) \\
&  \geq \PP\left( \text{$P_n$ and $P_1$ have the same number of zeros in }B\left(0,\frac{2|a_0|}{|a_1|} \right) \text{ and } \frac{2|a_0|}{|a_1|} \leq t \right) \\
& \geq \PP\left( \sup_{|z| = 2|a_0|/|a_1|} |P_n(z)-P_1(z)| < \inf_{|z| = 2|a_0|/|a_1|} |P_1(z)| \text{ and } 2\frac{|a_0|}{|a_1|} \leq t \right).
\end{align*}
We notice that the triangle inequality implies that
\[\inf_{|z| = 2|a_0|/|a_1|} |P_1(z)| = \inf_{|z| = 2|a_0|/|a_1|} |a_0 + a_1z| \geq \inf_{|z| = 2|a_0|/|a_1|} |a_1||z|-|a_0| = |a_0|\]
and that 
\begin{align*}
\sup_{|z| = 2|a_0|/|a_1|} |P_n(z)-P_1(z)|  \leq \sum_{k \geq 2} |a_k| \left(\frac{2|a_0|}{a_1}\right)^k 
\leq M \sum_{k\geq 2} e^k  \left(\frac{2|a_0|}{a_1}\right)^k \leq M \frac{4e^2 |a_0|^2/|a_1|^2}{1- 2e|a_0|/|a_1|}.
\end{align*}
Let $r$ be a constant such that $C_3=\PP(|a_1|>r)>0$, then we obtain
\begin{align*}
\PP(x_1^{(n)} \leq t) & \geq \PP\left( M \frac{4e^2 |a_0|^2/|a_1|^2}{1- 2e|a_0|/|a_1|} <  |a_0| \text{ and } 2\frac{|a_0|}{|a_1|} \leq t \right) \\
& \geq \PP\left( K \frac{4e^2 |a_0|^2/|a_1|^2}{1- 2e|a_0|/|a_1|} <  \frac{r|a_0|}{|a_1|}\text{ and } |a_1|>r \text{ and } 2\frac{|a_0|}{|a_1|} \leq t \text{ and } M<K \right)\\
\end{align*}
As $t<1/(A+1)$ then $A \frac{t^2}{1-t} < t$, we obtain that the event
\[\left\{2\frac{|a_0|}{|a_1|} \leq t \text{ and } |a_1|>r \right\} \text{ contains the event } \left\{K \frac{4e^2 |a_0|^2/|a_1|^2}{1- 2e|a_0|/|a_1|} <  r\frac{|a_0|}{|a_1|}\right\} \text{ if }t \leq \frac{e^{-1}}{K/r+1}.\] 

\noindent For $t \leq \frac{e^{-1}}{K/r+1}=C_1$, we get, using the independence of the $a_k$'s
\begin{align*}
\PP(x_1^{(n)} \leq t) & \geq \PP\left(2\frac{|a_0|}{|a_1|} \leq t \text{ and } |a_1|>r \text{ and } M<K \right) \\
& \geq \PP \left( |a_0|< \frac{rt}{2}\right) \PP\left(|a_1|>r\right) \PP\left(M<K\right)  \\
& \geq C_2C_3\PP \left( |a_0|< \frac{rt}{2}\right).
\end{align*}
\paragraph{Proof of 3.} Assume that there exists $k \geq 0$, $a>0$ and $\delta>0$ such that 
\[\forall \  t<\delta ,\quad \PP(|a_0| < t) \geq at^{k}.\]
Let $X$ be a non-negative random variable. Then the Fubini theorem implies that 
\[ \frac{1}{k+1}\EE(X^{k+1})= \int_0^{\infty} t^k \PP(X\geq t)dt.  \]
Using this along with the point $(1)$ we get
\begin{align*}
\frac{1}{k}\EE((x_n^{(n)})^{k}) & = \int_{0}^{\infty} t^{k-1} \PP(x_n^{(n)} \geq t) dt \\
& = \int_{0}^{\infty} t^{k-1} \PP(x_1^{(n)} \leq 1/t) dt \\
& \geq A\int_{1/C_1}^{\infty} t^{k-1} \PP(|a_0| \leq \frac{2}{rt}) dt \\
& \geq A  \frac{2^k}{r^k} \int_{1/C_1}^{\infty} a  t^{k-1} t^{-k} dt.
\end{align*}
This implies that
\[ \EE((x_n^{(n)})^{k}) = \infty. \]

\paragraph{Proof of 4.} Thanks to Lemma \ref{lemma}, we obtain that the radius of convergence of the random entire function $P_{\infty}$ is almost surely $1$. This implies that, almost surely, $P_n$ converges uniformly towards $P_{\infty}$ on any closed disk $\overline{D}(0,\rho)$ with radius $\rho<1$. In this setting, the almost sure convergence of the zeros of $P_n$ inside $\overline{D}(0,\rho)$ is exactly Hurwitz's theorem \cite[p.~152]{conway} in complex analysis. Hurwitz's theorem is a consequence of Rouché's theorem, which is a consequence of the argument principle. We give a proof of the convergence of the point processes using directly Rouché's theorem.

\noindent
Let $(\Omega, \mathcal{F},\PP)$ be a probability space on which the $a_k$'s are defined. Let $\mathcal{N}$ be a negligible set such that for any $\omega \in \Omega \setminus \mathcal{N}$, $P_{\infty}$ is a non-constant entire series with radius of convergence one.

Fix $\omega \in \Omega \setminus \mathcal{N}$ and let $z^{(\infty)}$ be a zero of $P_{\infty}$, with multiplicity $\beta$. As the zeros of $P_{\infty}$ are isolated, for any $\varepsilon$ small enough, $P_{\infty}$ has no other zero than $z^{(\infty)}$ in the closed disk $\overline{D}(z^{(\infty)},\varepsilon)$. Thanks to Rouché's theorem, we know that if
\begin{equation}\label{rouché}
\sup_{|z-z^{(\infty)}|=\varepsilon} |P_n(z)-P_{\infty}(z)| < \inf_{|z-z^{(\infty)}|=\varepsilon} |P_{\infty}(z)|
\end{equation}
then $P_n$ and $P_{\infty}$ have the same number of zeros inside $\overline{D}(z^{(\infty)},\varepsilon)$. The inequality \eqref{rouché} is automatically satisfied for $n$ large enough, as we fixed $\varepsilon$ such that $P_{\infty}$ does not have a zero on the boundary of the disk $D(z^{(\infty)}, \varepsilon)$. 

\noindent Here we proved that for any zero of multiplicity $\beta$ of $P_{\infty}$, for any $\varepsilon >0$ sufficiently small, one can find $\beta$ zeros of $P_n$ at a distance at most $\varepsilon$ of $z^{(\infty)}$. This implies that any fixed finite number of zeros of $P_n$ converges almost surely towards zeros of $P_{\infty}$.

\paragraph{Proof of 5.} The fact that $x_1^{(n)}$ converges almost surely toward $x_1^{(\infty)}$ is a consequence of the point 4). As we know that $x_n^{(\infty)}$ has the same distribution as $1/x_1^{(n)}$, then we obtain the converge in distribution of $x_n^{(\infty)}$ towards $1/x_1^{(\infty)}$ for free.
\end{proof}

\section{Comments.}
Notice that the proof of points $(1)$, $(2)$ and $(3)$ of Theorem \ref{heavytail} for finite $n$ does not use the assumption $\EE(\log(1+|a_0|))<\infty$. These three points are always valid. This assumption is only needed to make sure that, almost surely, the $a_k$'s do not grow faster than $e^k$.

\noindent
An alternative proof of Theorem \ref{heavytail} uses Jensen's formula \cite[p.~341]{lang} for analytic functions. We chose to use an approach based on Rouché's Theorem as it also implies the convergence of the point process of the small roots.

\bibliographystyle{alpha} 
\bibliography{biblio} 

\begin{thebibliography}{HKPV09}

\bibitem[BAZ98]{benarouszeitouni}
G{\'e}rard Ben~Arous and Ofer Zeitouni.
\newblock Large deviations from the circular law.
\newblock {\em ESAIM: Probability and Statistics}, 2:123--134, 1998.

\bibitem[BC14]{bordenavecaputo}
Charles Bordenave and Pietro Caputo.
\newblock A large deviation principle for wigner matrices without gaussian
  tails.
\newblock {\em The Annals of Probability}, 42(6):2454--2496, 2014.

\bibitem[BRS86]{bharucha}
Albert Bharucha-Reid and Masilamani Sambandham.
\newblock {\em Random Polynomials: Probability and Mathematical Statistics: a
  Series of Monographs and Textbooks}.
\newblock Academic Press, 1986.

\bibitem[But16]{butez}
Rapha{\"e}l Butez.
\newblock Large deviations for the empirical measure of random polynomials:
  revisit of the zeitouni-zelditch theorem.
\newblock {\em Electronic Journal of Probability}, 21, 2016.

\bibitem[BZ17]{butezzeitouni}
Rapha{\"e}l Butez and Ofer Zeitouni.
\newblock Universal large deviations for kac polynomials.
\newblock {\em Electronic Communications in Probability}, 22, 2017.

\bibitem[Con78]{conway}
John Conway.
\newblock {\em Functions of one complex variable}, volume~11 of {\em Graduate
  Texts in Mathematics}.
\newblock Springer-Verlag, New York-Berlin, second edition, 1978.

\bibitem[Ede97]{edelman1997proba}
Alan Edelman.
\newblock The probability that a random real gaussian matrix haskreal
  eigenvalues, related distributions, and the circular law.
\newblock {\em Journal of Multivariate Analysis}, 60(2):203--232, 1997.

\bibitem[Gin65]{ginibre}
Jean Ginibre.
\newblock Statistical ensembles of complex, quaternion, and real matrices.
\newblock {\em Journal of Mathematical Physics}, 6(3):440--449, 1965.

\bibitem[HKPV09]{houghtkris}
J.~Ben Hough, Manjunath Krishnapur, Yuval Peres, and B{\'a}lint Vir{\'a}g.
\newblock {\em Zeros of {G}aussian analytic functions and determinantal point
  processes}, volume~51 of {\em University Lecture Series}.
\newblock American Mathematical Society, Providence, RI, 2009.

\bibitem[HN08]{hughes}
Christopher Hughes and Ashkan Nikeghbali.
\newblock The zeros of random polynomials cluster uniformly near the unit
  circle.
\newblock {\em Compositio Mathematica}, 144(03):734--746, 2008.

\bibitem[HP00]{hiaipetz}
Fumio Hiai and D{\'e}nes Petz.
\newblock {\em The semicircle law, free random variables and entropy},
  volume~77.
\newblock American Mathematical Society Providence, 2000.

\bibitem[IZ13]{ibragimovzaporozhets}
Ildar Ibragimov and Dmitry Zaporozhets.
\newblock On distribution of zeros of random polynomials in complex plane.
\newblock In {\em Prokhorov and contemporary probability theory}, pages
  303--323. Springer, 2013.

\bibitem[KZ14]{kabluchkozapo}
Zakhar Kabluchko and Dmitry Zaporozhets.
\newblock Asymptotic distribution of complex zeros of random analytic
  functions.
\newblock {\em The Annals of Probability}, 42(4):1374--1395, 2014.

\bibitem[Lan13]{lang}
Serge Lang.
\newblock {\em Complex analysis}, volume 103.
\newblock Springer Science \& Business Media, 2013.

\bibitem[PV05]{peresvirag}
Yuval Peres and B{\'a}lint Vir{\'a}g.
\newblock Zeros of the iid gaussian power series: a conformally invariant
  determinantal process.
\newblock {\em Acta Mathematica}, 194(1):1--35, 2005.

\bibitem[SM08]{schehrmajumdar2008}
Gr{\'e}gory Schehr and Satya Majumdar.
\newblock Real roots of random polynomials and zero crossing properties of
  diffusion equation.
\newblock {\em Journal of Statistical Physics}, 132(2):235, 2008.

\bibitem[{\v{S}}{\v{S}}62]{sparosur}
DI~{\v{S}}paro and MG~{\v{S}}ur.
\newblock On the distribution of roots of random polynomials.
\newblock {\em Vestn. Mosk. Univ. Ser}, 1:40--43, 1962.

\bibitem[TW94]{tracywidom}
Craig Tracy and Harold Widom.
\newblock Level-spacing distributions and the airy kernel.
\newblock {\em Communications in Mathematical Physics}, 159(1):151--174, 1994.

\bibitem[ZZ10]{zeitounizelditch}
Ofer Zeitouni and Steve Zelditch.
\newblock Large deviations of empirical measures of zeros of random
  polynomials.
\newblock {\em Int. Math. Res. Not.}, (20):3935--3992, 2010.

\end{thebibliography}
\end{document}